\documentclass[12pt]{article} 

\usepackage{times}
\usepackage{latexsym}
\usepackage{amsmath}
\usepackage{amsfonts}
\usepackage{xcolor}
\usepackage{tikz}
\usepackage{amsthm}
\usepackage{parskip}

\begingroup
    \makeatletter
    \@for\theoremstyle:=definition,remark,plain\do{%
        \expandafter\g@addto@macro\csname th@\theoremstyle\endcsname{%
            \addtolength\thm@preskip\parskip
            }%
        }
\endgroup
\makeatletter
\newtheorem*{rep@theorem}{\rep@title}
\newcommand{\newreptheorem}[2]{%
\newenvironment{rep#1}[1]{%
 \def\rep@title{#2 \ref{##1}}%
 \begin{rep@theorem}}%
 {\end{rep@theorem}}}
\makeatother

\usepackage{epstopdf}
\newtheorem{Definition}[equation]{Definition}
\newcounter{mark}

\newtheorem{Theorem}[equation]{Theorem}
\newtheorem{Proposition}[equation]{Proposition}
\newtheorem{Lemma}[equation]{Lemma}
\newtheorem{Corollary}[equation]{Corollary}
\newtheorem{Remark}[equation]{Remark}

\newtheorem{Shufflep}[mark]{Shuffling Process}
\newtheorem{Question}[equation]{Question}

\definecolor{darkred}{rgb}{0.7,0,0} % darkred color
\newcommand{\darkred}{\color{darkred}} % darkred command
\newcommand{\defn}[1]{\emph{\darkred #1}} % emphasis of a definition

\newreptheorem{Theorem}{Theorem}
\newreptheorem{Mark}{Marking Scheme}

\newcommand{\eps}{\epsilon}
\newcommand{\f}[2]{\frac{#1}{#2}}
\renewcommand{\P}{\mathcal{P}}

\newcommand{\tme}{t_{\text{mix}}(\epsilon)}
\newenvironment{itemise}{\itemize}{\enditemize}
\newcommand{\M}{\mathcal{M}}

\title{A variation of strong stationary times for random walks with partial symmetries}
\author{Graham White}
\date{\today}
\begin{document}

\maketitle

\begin{abstract}
We introduce a variation of strong stationary times for random walks on the symmetric group. Rather than proceed in the usual fashion of accumulating larger and larger blocks of cards which may be in any order, we wait for pairs of cards to `interact', and bound the mixing time by the time taken for enough interactions to occur.
\end{abstract}
%\tableofcontents

\section{Introduction}
\label{sec:intro}

Strong stationary times are a powerful combinatorial technique for giving an upper bound on the mixing time of a random walk. Unfortunately, they can be very difficult to construct and often require that the random walk in question have quite a lot of symmetry. For instance, many strong stationary times for random walks on the symmetric group $S_n$ rely on looking at the positions of individual cards, and using induction to show that larger and larger sets of cards are uniformly random amongst themselves.

In this paper, we will introduce a similar technique which may be used for shuffling schemes which do not lend themselves to these kind of intermediate results. As an example, rather than inducting on larger and larger sets of random cards, we might track which pairs of cards have `interacted' with one another, in the sense that they might have swapped positions, and show that once all pairs of cards have interacted, the deck is random.

Once we have developed the necessary tools (Theorem \ref{the:genmutationfasttrans}), we show that they can easily give bounds of the correct order for several random walks of interest, including models of wash shuffles (Corollary \ref{cor:mutationfasttrans}), adjacent transpositions (Proposition \ref{prop:adjtransposition}), and a random walk on $S_n$ generated by a transposition and an $n$--cycle (Proposition \ref{prop:ncycletransposition}).

\subsection*{Acknowledgements}

I would like to thank my advisor, Persi Diaconis, for many helpful discussions and suggestions, about this work and others. 

\section{Strong stationary times}

The method of strong stationary times is a classical combinatorial technique (\cite{ADstrong}), where paths leading to various outcomes are placed in bijection with one another in order to show that these outcomes are equally likely. 

Strong stationary times have been used to give bounds on the separation distance mixing time of several random walks. A weakness of the technique seems to be that it is very difficult to use unless the state space is extremely structured, such as the symmetric group. Indeed, strong stationary times for walks on $S_n$ tend to `build up randomness', using the fact that inside $S_n$ there is an ascending chain of subgroups $$S_2 \leq S_3 \leq S_4 \leq \cdots \leq S_{n-1} \leq S_n.$$

Strong stationary times are used in \cite{Broder}, \cite{Matthews}, and \cite{GWtranspositions} to analyse the random transposition walk on the symmetric group. Chapter 4 of \cite{diaconis1988group} discusses strong stationary times further, discussing examples including the random transpositions of \cite{Broder} and \cite{Matthews}, and the top-to-random shuffle and random walk on the hypercube examples of \cite{ADstrong}. \cite{diaconis1990strong} presents examples and a construction of a dual Markov chain to a strong stationary time. 

A strong stationary time may be thought of as an exact description of a large set of paths, equally many of which end at each possible state of the chain. For example, in the random-to-top chain, where a step is to choose a random card and move it to the top of the deck, among paths of any length which choose every card at least once, equally many end at each possible order of the deck. 

Likewise in the inverse riffle shuffle chain, where a step consists of choosing a random set of cards and moving those cards to the top of the deck while preserving their relative order, paths of any given length where no two cards are moved by exactly the same set of steps are equally likely to end in any order of the deck. In contrast to the results of this paper, in each case those results completely specify the paths in question. 

While Proposition \ref{prop:mutationbound} gives bounds on the separation distance mixing time in a similar manner, it will not explicitly describe an equally-distributed set of paths in this way. In order to give such an upper bound on the mixing time, it is not necessary to exactly identify the set of paths, just to show that there are at least a certain number leading to each state.

Section \ref{sec:mutdifficulties} describes a shuffling scheme which seems as though it should be amenable to strong stationary time analysis, but where a reasonable attempt at constructing a strong stationary time fails, motivating our constructions. Section \ref{sec:mutdefinitions} gives definitions which will be used in later sections. Section \ref{sec:mutations} introduces mutation times and uses them to circumvent the issues encountered in Section \ref{sec:mutdifficulties}. Section \ref{sec:mutgeneral} uses mutation times to analyse some models of the wash shuffle, and Section \ref{sec:mutother} uses mutation times to derive upper bounds on the mixing times of classical random walks on the symmetric group $S_n$, including in cases where the construction of a strong stationary time is difficult.

Our focus is on presenting novel techniques rather than on obtaining optimal bounds. In particular, we will not be concerned with constant factors in the mixing time.

The arguments in this paper were inspired by \cite{DiaconisPal}, which introduces shuffling schemes such as Shuffle \ref{shu:1dwash1} and analyses them using similar methods. The exact relation between the techniques remains unclear. I am grateful to the authors of \cite{DiaconisPal} for discussing preliminary versions of their work with me.    

\section{Motivating example}
\label{sec:mutdifficulties}

A common method of shuffling cards is to spread them out on a table, and move them around with sweeping motions of the hands, which are in contact with multiple cards at once. This is done for a period of time, and then the cards are pushed together into one pile. As explained in \cite{DiaconisPal}, this is the standard method of shuffling in poker clubs and tournaments, and in Monte Carlo baccarat. Unlike riffle shuffles, which are well-modelled by the GSR shuffle, this `wash' shuffle is not well-understood.

Studying this shuffle via models of the following sort is due to Diaconis and Pal \cite{DiaconisPal}. The first model presented is a simple attempt at modelling the wash shuffle. It is introduced to study the essential difficulty considered in this section, and will be used in Section \ref{sec:mutations} as an initial example where a mutation time may be used to give an upper bound on the mixing time. In Section \ref{sec:mutgeneral} more sophisticated models will be considered, such as Shuffle \ref{shu:1dlongwash}.

\begin{Shufflep}
\label{shu:1dwash1}
Take a deck of $n$ cards, and place each card at one of the positions $1$ through $n$. At each step, choose a random card, and move it one step to the right, one step to the left, or leave it in place, with probabilities $\frac{1}{4}, \frac{1}{4}, $ and $\frac{1}{2}$. If a card moves to an occupied position, then it is inserted at a random position into the pile of cards already there.

After $t$ steps of this process, the piles are combined by placing the piles atop one another, with the pile from position $1$ on top.
\end{Shufflep}

This scheme is intended as the simplest possible model exhibiting the desired behaviour --- more faithful models would allow multiple cards to be moved at once, have cards be moved further than a single step, and perhaps operate in more than one dimension. We will consider some of these modifications in Section \ref{sec:mutgeneral}.

Analysis of this model is not difficult --- for instance, the position of a card depends only on how many times that card has been chosen to move, so once each card has been chosen enough times that its position is close to random, the resulting permutation will be close to random. This is because Shuffle \ref{shu:1dwash1} was defined so that when there are multiple cards at any position, those cards are equally likely to be in any order. The goal of this paper is to develop techniques applicable in situations where such simplifications are either not available, or give suboptimal bounds.

\begin{Definition}
In Shuffling Process \ref{shu:1dwash1} and similar processes in later sections, define two cards to \defn{interact} if they occupy the same position. (This is consistent with Definition \ref{def:interaction})
\end{Definition}

It is interesting to attempt to analyse Shuffle \ref{shu:1dwash1} via a strong stationary time, using the following observation. 

\begin{Remark}
\label{rem:ijswap}
Consider paths of length $t$ where cards $i$ and $j$ have interacted. Then for any permutation $\pi$, there are exactly as many of these paths ending at $\pi$ as ending at $\pi (i \; j)$ --- the permutation obtained from $\pi$ by transposing cards $i$ and $j$.
\end{Remark}
\begin{proof}
It suffices to give a bijection from paths ending at $\pi$ to paths ending at $\pi (i \; j)$. Let $\P$ be a path ending at $\pi$. Define $f(\P)$ to be the path which is the same as $\P$ until the first time the cards $i$ and $j$ are in the same position, and which from then onwards moves card $i$ whenever $\P$ would move card $j$ and vice versa. 
\end{proof}

Following Remark \ref{rem:ijswap}, one might think that the condition ``each pair of cards has interacted'' would be a strong stationary time, via repeated application of Remark \ref{rem:ijswap}, building arbitrary permutations as products of transpositions. Sadly, this is not the case. It can be checked explicitly that for $n=3$ cards and $t=3$ steps, conditioning on each pair of cards having interacted produces a nonuniform distribution on $S_3$.

The problem with this approach is that when a path is modified \`a la Remark \ref{rem:ijswap}, by interchanging the roles of two cards from some point forwards, this might destroy the property that each pair of cards has interacted. For instance if cards $i$ and $j$ met, and later $i$ and $k$ met, and finally $j$ and $k$ met, then swapping the roles of cards $i$ and $k$ from their meeting onwards would result in $i$ and $j$ meeting again rather than $j$ and $k$ meeting. The resulting path might no longer have each pair of cards interacting, which could invalidate attempts to apply Remark \ref{rem:ijswap} again.

It is possible to work around this problem, as in the following section. There will be two major differences. Bijections must be constructed not just between paths ending in $\pi$ and paths ending in $\pi (i \; j)$, but between paths ending in $\pi$ and paths ending in any other permutation $\pi'$. Also, it will no longer be the case that an exact set of equally-distributed paths is explicitly described by the stopping time in question --- rather, enough paths ending at each permutation will be shown to exist and described only implicitly.

For comparison purposes, the following remark is a version of Remark \ref{rem:ijswap} for an actual strong stationary time.

\begin{Remark}
Consider the random-to-top shuffle, where each step chooses a random card and moves it to the top of the deck. Then among paths of length $t$ which choose every card at least once, equally many paths produce each possible order of the deck.
\end{Remark}
\begin{proof}
Let $\P$ be a path which chooses each card at least once. For each permutation $\pi$, consider the path $\P_\pi$, defined by the rule that at each step, if path $\P$ moves card $i$, then the path $\P_\pi$ moves card $\pi(i)$. The collection of paths $\P_\pi$ contains one path which ends at each order of the deck.

Alternatively, rather than apply the whole permutation $\pi$ in one step, we could swap two cards, showing that any two permutations differing by a transposition are equally likely. Because this kind of modification --- moving each card instead of the other for the entire walk --- does not change whether or not each card has been moved at least once, this result may be iterated, unlike Remark \ref{rem:ijswap} for Shuffle \ref{shu:1dwash1}.  
\end{proof}

\section{Definitions}
\label{sec:mutdefinitions}
The techniques we use --- strong stationary times and variations --- give bounds on separation distance, so mixing times will be with regard to this distance.

The objects of study in this paper are shuffling processes. Rather than random walks on the symmetric group $S_n$, these are random processes in a larger state space. They are run for some number of steps and then the final state determines an element of $S_n$.

\begin{Definition}
\label{def:shufflingprocess}
A \defn{shuffling process} is a Markov chain together with a map from the state space to $S_n$.
\end{Definition}

For instance, a shuffling process might involve $n$ cards being moved around on a line or plane, and then collected in a specified order --- say, from left to right. While Definition \ref{def:shufflingprocess} does not require that the Markov chain have a finite state space, all of our examples will have finite state spaces. For examples with infinite state spaces, consider versions of Shuffles \ref{shu:1dwash1} or \ref{shu:1dlongwash} where the cards move on the set of integers rather than just the integers from $1$ to $n$.

The term `shuffling process' is reserved for the sense of Definition \ref{def:shufflingprocess}.

\begin{Definition}
The \defn{total variation distance} and \defn{separation distance} after $t$ steps of a shuffling process $\M$ have their usual meaning, with reference to the distributions on $S_n$ obtained by running $\M$ for $t$ steps and then mapping to $S_n$. The \defn{mixing time} of a shuffling process is defined as usual using either of these distances.
\end{Definition}

The motion of individual cards will not depend on their values.

\begin{Definition}
\label{def:fair}
Let $\M$ be a Markov chain on $S_n$. This chain is \defn{fair} if taking a step of $\M$ commutes with any permutation of card labels. 
\end{Definition}

For example, the operations of ``Reverse the deck'' and ``Perform a riffle shuffle'' are fair, while ``Move the ace of spades to the top'' and ``Shuffle until the bottom card is a face card'' are not. Indeed, any procedure on a deck of cards that would usually be described as shuffling is fair, because shuffling operations should not depend on the present positions of certain cards, but rather can be carried out with the cards being face down and indistinguishable to the shuffler. 

In particular, random walks on $S_n$ are fair.

\begin{Proposition}
Any random walk on $S_n$ is fair.
\end{Proposition}
\begin{proof}
Taking a step of a random walk is left-multiplication by a randomly chosen permutation. Permuting the card labels is right-multiplication by a permutation. These two operations commute.
\end{proof}

A fair shuffling process is defined in the natural way.

\begin{Definition}
\label{def:fair2}
Let $\M$ be a shuffling process equipped with an action of $S_n$ on its state space. This process is \defn{fair} if taking a step of $\M$ commutes with the $S_n$--action. 
\end{Definition}

In our settings, Definition \ref{def:fair2} will be applied to processes involving $n$ cards labelled $1$ to $n$, and the $S_n$--action will always be permutation of these labels. All processes considered will be fair --- this condition is used in Theorems \ref{the:genmutationslowtrans} and \ref{the:genmutationfasttrans}.

\begin{Definition}
\label{def:interaction}
Let $\M$ be a fair Markov chain or shuffling process with an action of $S_n$ as in Definition \ref{def:fair2}. Let $x$ and $y$ be states of $\M$, and $i$ and $j$ be integers between $1$ and $n$. Say that $i$ and $j$ \defn{interact at $y$} if $y$ is fixed by the action of the transposition $(i \; j)$. Say that $i$ and $j$ \defn{interact between $x$ and $y$} if the probabilities $\M(x,y)$ and $\M(x,y\cdot(i \; j))$ are equal.
\end{Definition}

It will be necessary to discuss the time at which interactions take place.

\begin{Definition}
\label{def:interactiontime}
Continuing from Definition \ref{def:interaction}, let the state at each time $t$ be $X_t$. Say that $i$ and $j$ \defn{interact at $t$} if they interact at $X_t$, and that they \defn{interact between $t$ and $t+1$} if they interact between $X_t$ and $X_{t+1}$.
\end{Definition}

Definition \ref{def:interaction} is motivated by shuffling processes such as Shuffles \ref{shu:1dwash1}, \ref{shu:1dlongwash} and \ref{shu:ddwash1}, which are described later. In those processes, interaction corresponds to the two cards in question moving to the same location. Notice that the second part of Definition \ref{def:interaction} generalises the first --- if $i$ and $j$ interact at $y$, then they interact between $x$ and $y$ for any $x$. 

The key property of these definitions is that for any path in a fair Markov chain or shuffling process $\M$, if $i$ and $j$ interact at time $t$, then the path may be modified by exchanging $i$ and $j$ at all times larger than $t$, and the resulting path is equally probable. The same is true if $i$ and $j$ interact between times $t$ and $t+1$, which is the motivation for the second part of Definition \ref{def:interaction}.

It will be common for us to show that a set of cards are equally likely to be in any order amongst themselves, and that this is independent of the positions of other cards.

\begin{Definition}
\label{def:jumbled}
Given a distribution on $S_n$, a set of labels $A$ is \defn{jumbled} if the distribution is unchanged under permuting the cards of those labels. 
\end{Definition}

Note that this definition requires not only that the set of cards $A$ are equally likely to be in any order amongst themselves, but also that this is independent of the positions of all other cards. 

\section{Mutation times}
\label{sec:mutations}

This section introduces mutation times and shows how they may be used to give upper bounds on the separation distance mixing time. Shuffle \ref{shu:1dwash1} is used as a running example, with two different mutation times given in Propositions \ref{prop:mutationslowtrans} and \ref{prop:mutationfasttrans}. These mutation times are generalised to other shuffling processes in Section \ref{sec:mutgeneral}.

\begin{Definition}
\label{def:mutationtime}
Consider a Markov chain on the symmetric group $S_n$. A stopping time $T$ is a \defn{mutation time} if for each initial state $x_0$ and each pair of permutations $\pi$ and $\pi'$, there is an injection from paths from $x_0$ of length $t$ satisfying $T$ and ending in $\pi$ to paths from $x_0$ of length $t$ ending in $\pi'$ (these paths need not satisfy $T$).
\end{Definition}

Mutation times give upper bounds on mixing in much the same way as strong stationary times do, except that rather than explicitly identifying a large multiple of the uniform distribution, they just show its existence.

\begin{Proposition}
\label{prop:mutationbound}
For a Markov chain on $S_n$ with uniform stationary distribution, if $T$ is a mutation time with $P(T > t) < \eps$ then $\tme \leq t$.
\end{Proposition}
\begin{proof}
Consider all paths of length $t$ satisfying $T$. These paths have total probability at least $1 - \eps$. Therefore there is a permutation $\pi$ such that the probability of a path of length $t$ both satisfying $T$ and ending at $\pi$ is at least $\frac{1 - \eps}{n!}$. 

Because $T$ is a mutation time, for each permutation $\pi'$, there is an injection from these paths to paths of length $t$ which end at $\pi'$ (but which do not necessarily satisfy $T$). Therefore among paths of length $t$, at least a proportion $\frac{1-\eps}{n!}$ end at $\pi'$, for each permutation $\pi'$. Hence after $t$ steps, the separation distance from the uniform distribution is at most $\epsilon$, as claimed.
\end{proof}

If Definition \ref{def:mutationtime} were only required for certain starting states, then Proposition \ref{prop:mutationbound} would give bounds on the mixing of a chain started from any of those states.

An initial example of a mutation time is as follows. It is not intended to be optimal, but rather to illustrate the principle. Proposition \ref{prop:mutationfasttrans} defines an improved mutation time. 

\begin{Proposition}
\label{prop:mutationslowtrans}
As a deck of $n$ cards is shuffled via Shuffling Process \ref{shu:1dwash1}, define a stopping time as follows. Let $T_0 = 0$, and for each $i$, let $T_i$ be the earliest time such that between times $T_{i-1}$ and $T_i$, card $i$ has interacted with each other card. Then $T_n$ is a mutation time.
\end{Proposition}

The following lemma will be necessary.

\begin{Lemma}
\label{lem:expression1}
Any element $\pi$ of $S_n$ has an expression of the form $$\pi = \prod_{i=1}^n (i \; a_i),$$ where each $a_i$ is between $1$ and $n$. 
(If $a_i$ is taken to be between $1$ and $i$ then these expressions are unique, and this is the standard algorithm for generating a uniform random permutation. The weaker result is used for simplicity of the example.) 
\end{Lemma}
\begin{proof}
Assume that this result is true for $n-1$, and consider expressions of the form $$\prod_{i=1}^n (i \; a_i).$$ As $a_1$ through $a_{n-1}$ range between $1$ and $n-1$, the product of the first $n-1$ terms ranges over all elements of $S_{n-1}$, and then each choice of $a_n$ produces one of the $n$ cosets of $S_{n-1}$ in $S_n$.

The lemma is true for $n=2$, completing the proof.  
\end{proof}

\begin{proof}[Proof of Proposition \ref{prop:mutationslowtrans}]
To show that $T_n$ is a mutation time, it is required to, for each pair of permutations $\pi$ and $\pi'$, define a map from paths of length $t$ which end in $\pi$ and satisfy $T_n$ to paths of length $t$ which end in $\pi'$.

The basic operation used to define these maps will be to look at an instance of cards $i$ and $j$ interacting, in the sense of Definition \ref{def:interaction}, and to, at that time, swap them. This does not involve exchanging their current positions --- the two cards are presently in the same place --- but rather, to exchange their roles in the remaining part of the path. That is, to alter the path by decreeing that after this time, card $i$ will move in place of card $j$, and vice versa. This new path will have the same probability as long as the shuffling process in question is fair, which Shuffle \ref{shu:1dwash1} is.

Given a path that ends at $\pi$, to have it instead end at $\pi'$ requires that the permutation $\pi'\pi^{-1}$ be applied to the labels. From Lemma \ref{lem:expression1}, if one is allowed to swap card $1$ with any card, and then card $2$ with any card, and so on, then it is possible to effect any desired permutation. The stopping time $T_n$ is designed to allow exactly this.

Thus, for fixed $\pi$ and $\pi'$, the desired function from paths ending at $\pi$ to paths ending at $\pi'$ is defined as follows.
\begin{itemise}
\item Consider a path ending at $\pi$ which satisfies $T_n$.
\item Let $a_1$ through $a_n$ be the constants used to write the permutation $\pi'\pi^{-1}$ as in Lemma \ref{lem:expression1}.
\item For each $i$, take the first time between $T_{i-1}$ and $T_i$ when cards $i$ and $a_i$ interact. At that time, swap those cards.
\item The resulting path ends at $\pi'$.
\end{itemise}

All that remains is to check that this map is an injection. The following procedure is a well-defined inverse. Note that the constants $a_1$ through $a_n$ depend only on the permutations $\pi$ and $\pi'$, and not on the path. Care is needed because swapping two cards changes which cards interact later in the path, and so the times $T_i$ cannot be read directly from the modified path.

Set $T_0 = 0$. For each $i$ from $1$ to $n$, wait until $i$ interacts with $a_i$ after time $T_{i-1}$. At this time, swap those two cards. Then let $T_i$ be the time at which card $i$ has interacted with each other card after time $T_{i-1}$.

This completes the proof.
\end{proof}

\begin{Corollary}
\label{cor:mutationslowtrans}
The mixing time $\tme$ of Shuffle \ref{shu:1dwash1} is at most $O(n^4)$.
\end{Corollary}
\begin{proof}
From Proposition \ref{prop:mutationbound}, it suffices to show that the mutation time $T_n$ of Proposition \ref{prop:mutationslowtrans} is at most $O(n^4)$ with probability at least $1-\eps$.

A sufficient condition for a card to interact with each other card is for it to move to both the far-right and far-left positions. This takes $O(n^2)$ steps of a simple random walk, or $O(n^3)$ steps of Shuffle \ref{shu:1dwash1}, where any given card is chosen with probability $\f1n$. Therefore each $T_i - T_{i-1}$ takes at most $O(n^3)$ steps, and so $T_n$ is at most $O(n^4)$.

Only rough analysis is given here because the mutation time of Proposition \ref{prop:mutationslowtrans} is far from optimal.
\end{proof}

The mutation time of Proposition \ref{prop:mutationslowtrans} had the advantage of relying on very simple factorisations of permutations from Lemma \ref{lem:expression1}, but it turns out that considering only one card at a time is rather wasteful.

\begin{Proposition}
\label{prop:mutationfasttrans}
As a deck of $n$ cards is shuffled via Shuffling Process \ref{shu:1dwash1}, let $T$ be the earliest time at which each pair of cards has interacted. Then $T$ is a mutation time.
\end{Proposition}

The proof of this result will need the following lemmas. Lemma \ref{lem:transpositiongen} is standard, see for instance \cite{diaconis1977spearman} and \cite{diaconis2004poisson} for the respective parts.

\begin{Lemma}
\label{lem:transpositiongen}
Consider the symmetric group $S_n$ generated by the set of all permutations $(i \; j)$ and let $l$ be the corresponding length function ($l$ is often called the \defn{Cayley distance}). Then 
\begin{enumerate} 
\item The length of a permutation $\pi$ is $$l(\pi) = n - \#(\text{cycles in }\pi).$$
\item If $\pi$ is a permutation and $(i \; j)$ is a transposition, then $l((i \; j)\pi) = l(\pi) - 1$ if and only if $i$ and $j$ are in the same cycle in $\pi$. Otherwise $l((i \; j)\pi) = l(\pi) + 1$. The same is true when $\pi$ is multiplied on the right by $(i \; j)$, rather than on the left.
\end{enumerate} 
\end{Lemma}
\begin{proof}
Both parts follow from noting that for any two disjoint cycles $C_1 = (i \; a_1 \; a_2 \; \dots \; a_p)$ and $C_2 = (j \; b_1 \; b_2 \; \dots \; b_q)$, $$(i \; j)C_1C_2 = (j \; b_1 \; \dots \; b_q \; i \; a_1 \; \dots \; a_p).$$ Multiplying on the left by $(i \; j)$, it is also true that $$C_1C_2 = (i \; j)(j \; b_1 \; \dots \; b_q \; i \; a_1 \; \dots \; a_p).$$ The first of these equations describes multiplication by $(i \; j)$ when $i$ and $j$ are in different cycles, and the second when they are in the same cycle. 

Changing from left-multiplication to right- is conjugation by $(i \; j)$, which does not change whether or not $i$ and $j$ are in the same cycle.
\end{proof}

\begin{Lemma}
\label{lem:expression2}
Let $s_1$ through $s_{\binom{n}{2}}$ be transpositions in $S_n$, with each different transposition occurring exactly once. Then any permutation $\pi$ may be expressed as the product of some subsequence of the $(s_i)$ (in their original order).

That is, for any permutation $\pi$ there is a sequence $\eps_1$ through $\eps_{\binom{n}{2}}$ of $0$s and $1$s so that $\pi$ may be written as \begin{equation}\label{eq:expression2}\pi = \prod_{i=1}^{\binom{n}{2}}s_i^{\eps_i}.\end{equation}
\end{Lemma}
\begin{proof}
Fix the sequence of $(s_i)$. Consider starting with the permutation $\pi$ and being offered in turn the chance to multiply on the left by $s_1$ through $s_{\binom{n}{2}}$. Noting that each $s_i$ is its own inverse, it suffices to show that from any $\pi$ it is possible to reach the identity. Framing the question in terms of a path from $\pi$ to the identity also means that the $s_i$ are now multiplied on the left. Essentially, taking the inverse of Equation \ref{eq:expression2} has reversed the order of factors on the right side.

To move from $\pi$ to the identity by multiplying by some subsequence of the $s_i$, use exactly those $s_i$ that reduce the length of the present permutation, where length $l$ is with respect to the generating set of all transpositions as in Lemma \ref{lem:transpositiongen}. 

That is, set $\eps_{k}$ to be $1$ exactly if $$l\left(s_{k}\prod_{i=k-1}^{1}s_i^{\eps_i}\pi\right) < l\left(\prod_{i=k-1}^{1}s_i^{\eps_i}\pi\right).$$

Notice that the cumulative products $\prod_{i=k}^{1}s_i^{\eps_i}\pi$ only ever decrease in length, forming (potentially part of) a geodesic from $\pi$ to the identity in the Cayley graph of $S_n$ with respect to this generating set. Thus if two elements are in different cycles in one of these products $\prod_{i=k}^{1}s_i^{\eps_i}\pi$, then they are in different cycles of all future products $\prod_{i=k+m}^{1}s_i^{\eps_i}\pi$.

For the sake of a contradiction, assume that this procedure does not end at the identity. If the final product $$\pi' = \prod_{i=\binom{n}{2}}^{1}s_i^{\eps_i}\pi$$ is not the identity, then $\pi'$ has a nontrivial cycle. Let $p$ and $q$ be in the same cycle in $\pi'$. But then $p$ and $q$ were in the same cycle in all previous products, meaning that when $s_i$ was the transposition $(p \; q)$, $\eps_i$ should have been chosen to be $1$, which would have separated $p$ and $q$. This is a contradiction, so the specified choice of the $(\eps_i)$ always produces the identity element.
\end{proof}
\begin{proof}[Proof of Proposition \ref{prop:mutationfasttrans}]
As in Proposition \ref{prop:mutationslowtrans}, it is required to, for each pair of permutations $\pi$ and $\pi'$, define a map from paths of length $t$ which end in $\pi$ and satisfy $T$ to paths of length $t$ which end in $\pi'$. A path ending in $\pi$ will be mutated into a path ending in $\pi'$ by applying the permutation $\pi'\pi^{-1}$, acting on labels rather than positions. This permutation will be applied by swapping cards when they interact, with the definition of the stopping time $T$ guaranteeing a certain arrangement of available swaps, and Lemma \ref{lem:expression2} showing that these swaps are enough to create any permutation $\pi'\pi^{-1}$.

If a path satisfies $T$, then each pair of cards has interacted. Let $s_i$ be the transposition of the two cards which were the $(\binom{n}{2} + 1 - i)$th pair to interact, considering only the last interaction between any pair of cards. If multiple interactions happen simultaneously, order them arbitrarily. Reversing the order ensures that multiplication is done on the correct side.

In order to apply the permutation $\pi'\pi^{-1}$ to the labels, write $\pi'\pi^{-1}$ as a product of the $s_i$ as in Lemma \ref{lem:expression2}, and for each $s_i$ appearing in this expression, swap the corresponding two cards at the first time at which they interact.

Applying $\pi'\pi^{-1}$ to the labels of $\pi$ gives $\pi'$, so this modified path ends at $\pi'$. All that remains is to show that the map is injective. This is less immediate than in the proof of Proposition \ref{prop:mutationslowtrans}, because the factorisation of Lemma \ref{lem:expression2} is not as explicit than that of Lemma \ref{lem:expression1}. However, the proof of Lemma \ref{lem:expression2} does describe how to choose the $s_i$, and this is enough to undo the map. Given a path ending at $\pi'$ produced in this way, consider each step in the path in reverse order. If a step is the last interaction between two cards such that swapping them would shorten $\pi'\pi^{-1}$ as in Lemma \ref{lem:expression2}, then swap them. 

One complication mostly avoided by starting at the end of the path is that exchanging two cards also changes which cards interact later in the path. Thus, one cannot write down the entire sequence of $(s_i)$ initially, but must proceed one step at a time, using the current intermediate value of the path to evaluate whether or not an interaction between two cards is the final one. (Starting instead from the beginning of the path would mean that each time a transposition was used, the remainder of the path would be conjugated by that transposition. In effect, this warns that one is multiplying terms in the wrong order).
\end{proof}

\begin{Corollary}
\label{cor:mutationfasttrans}
The mixing time $\tme$ of Shuffle \ref{shu:1dwash1} is at most $O(n^3\log(n))$.
\end{Corollary}
The following lemma will be useful
\begin{Lemma}
\label{lem:combininglog}
If $(T_i)$ are a collection of $k$ stopping times, each with expectation at most $t_0$, then for any fixed $\eps > 0$, the time until there is a probability of at least $1-\eps$ that each $T_i$ is satisfied is at most $O(t_0\log(k))$.
\end{Lemma}
\begin{proof}
After each $2t_0$ steps, each of the $T_i$ has at least a probability of $\f12$ of being satisfied, by Markov's inequality. Therefore the expected number of the $T_i$ which are as yet unsatisfied is multiplied by at most $\f12$ after each $2t_0$ steps. After $2t_0\log_2(\f{k}{\eps})$ steps, this expectation is below $\eps$, so the probability that any of the $T_i$ are unsatisfied is at most $\eps$. This completes the proof. 
\end{proof} 
\begin{proof}[Proof of Corollary \ref{cor:mutationfasttrans}]
As in the proof of Corollary \ref{cor:mutationslowtrans}, it takes $O(n^3)$ steps for any given card to meet each other card. Lemma \ref{lem:combininglog} gives that the time required for each card to have done this is $O(n^3\log(n))$.
\end{proof}

\section{More general shuffling processes}
\label{sec:mutgeneral}

The mutation times of the previous section may be applied to much more general procedures than Shuffle \ref{shu:1dwash1}. This section will give examples which modify Shuffle \ref{shu:1dwash1} to include features of actual wash shuffles --- that cards may move a long distance in a single step, and that the shuffle may take place in dimension greater than one. Section \ref{sec:mutother} will apply these techniques to some classical random walks on $S_n$, unrelated to wash shuffles.

\begin{Theorem}
\label{the:genmutationslowtrans}
Let $\M$ be a fair shuffling process, as in Definition \ref{def:fair2}. Define a stopping time as follows. Let $T_0 = 0$, and $T_i$ be the earliest time such that between times $T_{i-1}$ and $T_i$, for each $j$ between $1$ and $n$, $i$ has interacted with $j$. Then $T_n$ is a mutation time for $\M$.
\end{Theorem}

\begin{Theorem}
\label{the:genmutationfasttrans}
Let $\M$ be a fair shuffling process, and $T$ be the earliest time at which each pair $(i,j)$ has interacted. Then $T$ is a mutation time for $\M$.
\end{Theorem}
\begin{proof}
The proofs are the same as those of Propositions \ref{prop:mutationslowtrans} and \ref{prop:mutationfasttrans}, where any reference to swapping two cards $i$ and $j$ is replaced by `acting by the transposition $(i \; j)$'.
\end{proof}

Theorems \ref{the:genmutationslowtrans} and \ref{the:genmutationfasttrans} may be used to give bounds on the mixing times of a variety of shuffling processes. These bounds will not necessarily be tight, but they do occur as relatively swift applications of the techniques of that section. The goal of this section is to show how these theorems may be applied, rather than to produce optimal bounds.

The first example is a model for wash shuffles in one dimension.

\begin{Shufflep}
\label{shu:1dlongwash}
Take a deck of $n$ cards, and place each card at one of the positions $1$ through $n$. At each step, each card moves a random number of places to the right, distributed as independent geometric random variables with mean $\frac1p$, and then each card moves a random number of steps to the left, distributed in the same way. 

Movement happens position-by-position, with each card moving from position $1$ moving to position $2$ simultaneously, and then any of those cards moving more than one position, together with any cards moving from position $2$, moving to position $3$, and so on. When a pile of cards is moved into an occupied position, it is merged with the pile of cards already there by a GSR riffle shuffle. Cards that would be moved beyond position $1$ or $n$ move to that position and then stop.

After $t$ steps, piles are collected in order.
\end{Shufflep}

A step of this shuffling process should be understood as a hand sweeping from left to right, dragging each card some distance before releasing it, and then doing the same from right to left. The assumption that piles are combined by GSR riffle shuffles ensures that each pile of cards is jumbled, in the sense of Definition \ref{def:jumbled}, and it is unclear how reasonable an assumption this is for actual wash shuffles. 

\begin{Proposition}
After any number of steps of Shuffle \ref{shu:1dlongwash}, each pile of cards is jumbled. 
\end{Proposition}
\begin{proof}
Because the cards move independently, it suffices by induction to show that when two jumbled piles of cards are combined by a riffle shuffle, the resulting pile is jumbled. When two piles of size $a$ and $b$ are combined using a riffle shuffle, the set of cards in the resulting pile of size $a+b$ which came from the pile of size $a$ is equally likely to be any of the $\binom{a+b}{a}$ subsets of size $a$. Therefore if the initial two piles were jumbled, then so is the combined pile.
\end{proof}

Another assumption of Shuffle \ref{shu:1dlongwash} is that each card is moved in the same way --- a more realistic assumption might be that cards near the top of their pile are moved further, or are the only cards moved.

Note that in Shuffle \ref{shu:1dlongwash}, $p$ is a probability, with $\frac{1}{p}$ controlling how far cards move at each step. For instance, if cards are swept on average $\sqrt{n}$ places, then $p = \frac{1}{\sqrt{n}}$. 

Shuffle \ref{shu:1dlongwash} was defined using GSR riffle shuffles to combine piles of cards. The following is an equivalent definition which aids in understanding the shuffling process, but does not appear as similar to a mash shuffle.

\begin{Remark}
\label{rem:washjumble}
Instead of moving cards position-by-position and combining piles with riffle shuffles, Shuffle \ref{shu:1dlongwash} could have moved cards one at a time, inserting them at a random position in the pile at their destination. This shuffling process is equivalent.
\end{Remark}
\begin{proof}
The original formulation of Shuffle \ref{shu:1dlongwash} has the property that the pile of cards at any position is jumbled. This procedure has the same property, and the same distribution of the positions of cards. Therefore they are equivalent.
\end{proof}

The following lemma will be needed for the analysis of Shuffle \ref{shu:1dlongwash}.

\begin{Lemma}
\label{lem:washinteraction}
In Shuffle \ref{shu:1dlongwash}, let the positions of cards $i$ and $j$ at time $t$ be $I_t$ and $J_t$. Let $I_{t+0.5}$ and $J_{t+0.5}$ be the positions that cards $i$ and $j$ occupy after they are moved to the right but before they have been moved to the left. If $I_t < J_t$ and $I_{t+0.5} \geq J_{t+0.5}$, then $i$ and $j$ interact at time $t+1$, in the sense of Definition \ref{def:interactiontime}.
\end{Lemma}
\begin{proof}
It suffices to check that the probability that the first half of step $(t+1)$ moves cards $i$ and $j$ to $I_{t+0.5}$ and $J_{t+0.5}$ is the same as the probability that it moves them to $J_{t+0.5}$ and $I_{t+0.5}$ instead, because cards move independently of one another. By Remark \ref{rem:washjumble}, it is not necessary to consider the order of cards within piles, just the position of each card. Let these probabilities be $P_\pi$ and $P_{\pi(i \; j)}$. Then \begin{align*}
P_\pi &= (1-p)^{I_{t+0.5}-I_t}p(1-p)^{J_{t+0.5}-J_t}p \\ &= (1-p)^{I_{t+0.5} + J_{t+0.5}-I_t-J_t}p^2 \\ &= (1-p)^{J_{t+0.5}-I_t}p(1-p)^{I_{t+0.5}-J_t}p \\ &= P_{\pi(i \; j)}.
\end{align*}

This completes the proof.
\end{proof}

In exactly the same way, two cards interact when one is moved past the other while they are moving to the left.

\begin{Proposition}
The mixing time of Shuffle \ref{shu:1dlongwash} is at most $O(p^2n^2\log(n))$.
\end{Proposition}
\begin{proof}
Consider the mutation time $T$ given in Theorem \ref{the:genmutationfasttrans}. By Lemma \ref{lem:washinteraction}, two cards have interacted when one moves to the position of the other or overtakes it. Thus when two cards have interacted in this way, they may be swapped from there on without changing the probability of the path.

To bound the mutation time $T$, for any fixed $i$ and $j$, consider the (signed) distance $d_{ij}$ between cards $i$ and $j$. The distance is initially $j-i$. At each step, the expected change in this distance is zero, and the variance of the change is $2\f{1-p}{p^2}$. Thus at each step, $d_{ij}$ evolves similarly to $2\f{1-p}{p^2}$ steps of a simple random walk on the interval $[0,n]$, which requires $O(n^2)$ steps to hit $0$. Therefore the expected time for cards $i$ and $j$ to interact is at most $O(n^2\f{p^2}{2(1-p)})$. For small $p$, this is $O(n^2p^2)$. 

Lemma \ref{lem:combininglog} gives that the time for each pair of cards to interact is at most $O(p^2n^2\log(n))$.
\end{proof}

For instance, when $p = \frac{1}{\sqrt{n}}$, the mixing time of Shuffle \ref{shu:1dlongwash} is at most $O(n\log(n))$. This bound depends heavily on $p$. If $p = \frac{2}{n}$, so cards are moved on average half of the length of the interval, then the mixing time of Shuffle \ref{shu:1dlongwash} is at most $O(\log(n))$.

It is possible to analyse analogues of Shuffle \ref{shu:1dwash1} in higher dimensions.

\begin{Shufflep}
\label{shu:ddwash1}
Let $G$ be a $d$--dimensional grid with each side length $n$. Take a deck of $n$ cards, and place each card at an arbitrary vertex of $G$. At each step, choose a random card. With probability $\frac{1}{2}$ leave it in place, otherwise move it to a uniformly chosen neighbouring vertex. If a card moves to an occupied position, then it is inserted at a random position into the pile of cards already there.

After $t$ steps of this process, the piles are combined by placing the piles atop one another, in lexicographic order.
\end{Shufflep}

\begin{Proposition}
The mixing time of Shuffle \ref{shu:ddwash1} is at most $O(n^{2d+1}\log(n))$.
\end{Proposition}
\begin{proof}
For any two cards $i$ and $j$, each time either card is chosen the difference between their positions takes one step of a simple random walk on a $d$--dimensional square grid of side length $2n$ (The difference will sometimes be unable to move near the outside of this grid, when one of the cards in question hits an edge, but this only decreases the time to hit zero). From \cite{LAWLER198685}, the expected number of these steps until the difference becomes zero is at most $O(n^{2d})$. This takes time $O(n^{2d+1})$, because there are $n$ different cards which could be moved at any step. There are $\binom{n}{2}$ pairs $(i,j)$, so Lemma \ref{lem:combininglog} completes the proof.
\end{proof}

\section{Applications to classical shuffles}
\label{sec:mutother}

The mutation times defined in the previous section may also be used to obtain bounds on several classical random walks on $S_n$. 

\begin{Proposition}
\label{prop:adjtransposition}
The mixing time of the lazy adjacent transposition walk is at most $O(n^3\log(n))$ (at each step swap two adjacent cards with probability $\frac{1}{2}$, else do nothing).
\end{Proposition}
\begin{proof}
Consider two cards $i$ and $j$. The difference between their positions performs a simple random walk on $[0,n-1]$, taking a step whenever either of them is moved (this happens with probability approximately $\f2n$ and at least $\f1n$, depending on whether either of the cards are at the extreme positions $1$ or $n$. The expected number of steps for this walk to hit $0$ is $O(n^2)$. Therefore, the expected time until cards $i$ and $j$ interact is $O(n^3)$. Lemma \ref{lem:combininglog} gives that the time until each pair of cards have interacted is $O(n^3\log(n))$, and Theorem \ref{the:genmutationfasttrans} completes the proof.
\end{proof}

The bound of Proposition \ref{prop:adjtransposition} is of the correct order. The upper bound is obtained using comparison theory in Section 4, Example 2 of \cite{Compgroups}, or using coupling as Example 4.10 of \cite{aldous1983}. The lower bound is derived using Wilson's method in \cite{wilson2004mixing}.

\begin{Proposition}
\label{prop:ncycletransposition}
Consider the random walk on $S_n$ generated by the identity, the $n$--cycle $(1 \; 2 \; \dots \; n)$, and the transposition $(1 \; 2)$. This walk has mixing time at most $O(n^3\log(n))$.
\end{Proposition}
\begin{proof}
Consider the (cyclic) distance between two cards $i$ and $j$. After each $3n$ steps, approximately, each card has been moved to the top of the deck once. Each time a card moves to the top of the deck, there is a chance of at least $\f13$ that it will take at least one step in either direction that is not immediately undone. Thus, the distance between $i$ and $j$ mixes at least as fast as a simple random walk on a cycle of length $n$ which takes a step every $9n$ steps of the original walk. That simple random walk needs $O(n^2)$ steps to hit zero, so the expected time for $i$ and $j$ to interact is at most $O(n^3)$. Lemma \ref{lem:combininglog} and Theorem \ref{the:genmutationfasttrans} complete the proof.
\end{proof}

Proposition \ref{prop:ncycletransposition} produces a mixing time of the correct order for this random walk. The upper bound is obtained from comparison theory in Section 3, Example 3 of \cite{Compgroups}, and the lower bound from Wilson's method in \cite{wilson2004mixing}.

\begin{Proposition}
\label{prop:mutrtr}
The mixing time of the random-to-random shuffle (at each step, a random card is removed and then inserted at a random position) is at most $O(n^2\log(n))$.
\end{Proposition}
\begin{proof}
At each step, the card which was moved interacts with the card on top of which it is placed, because it could instead have been placed one position lower. Other interactions occur, but for simplicity, only consider interactions of this form.

The expected time for cards $i$ and $j$ to interact is $O(n^2)$ steps. Lemma \ref{lem:combininglog} and Theorem \ref{the:genmutationfasttrans} complete the proof. 
\end{proof}

The bound of Proposition \ref{prop:mutrtr} is of the wrong order --- this mixing time is known to be $O(n\log(n))$. The total variation mixing of this walk is analysed in \cite{MErtr}.

\section{Further work}

When Theorem \ref{the:genmutationfasttrans} is used as in Proposition \ref{prop:mutrtr}, there may be room for improvement in Lemma \ref{lem:expression2}. The theorem says that waiting for each pair of cards to interact suffices, but it isn't clear whether it is possible to improve this criterion. For instance, this approach gave a bound of the wrong order for the random-to-random walk (Proposition \ref{prop:mutrtr}). It would be interesting to improve this bound by either identifying more opportunities for interactions in this walk, or reducing the required interactions. 

\begin{Question}
Given a sequence of random transpositions, if it is long enough that it contains each transposition at least once, then it has the property of Lemma \ref{lem:expression2}. Is it likely for significantly shorter sequences to have this property? The time until Lemma \ref{lem:expression2} is satisfied is the coupon collector time of $\frac12n^2\log(n)$, but most of this time is spent waiting for the last few transpositions. Are those necessary, or can the repeats of other transpositions take their place?
\end{Question}

Theorem \ref{the:genmutationfasttrans} shows that waiting for each pair of cards to interact suffices for a mutation time. It is possible that this is a special case of a more general result about groups and generating sets. 

\begin{Question}
Theorem \ref{the:genmutationfasttrans} pertains to the symmetric group as generated by transpositions (Definition \ref{def:interaction}). Are there generalisations of this theorem for other generating sets? For instance, perhaps a random walk might offer opportunities for three cards to interact at once --- how many of these interactions are required? What about a `shuffling process' which results in an element of a group other than $S_n$ with an analogue of (Definition \ref{def:interaction}) for a `nice' generating set for that group --- is there an analogue of Theorem \ref{the:genmutationfasttrans}?
\end{Question}

\bibliographystyle{plain}
\bibliography{bib}
\end{document}